\documentclass[12pt,amstex]{amsart}

\usepackage{mathptmx}
\usepackage{mathrsfs}

\usepackage{verbatim}
\usepackage{url}
\usepackage[all]{xy}
\usepackage{color}

\usepackage[colorlinks=true,citecolor=blue]{hyperref}

\usepackage{stmaryrd}
\usepackage{epsfig}
\usepackage{amsmath}
\usepackage{amssymb}

\usepackage{amscd}
\usepackage{graphicx}
\usepackage{pstricks}
\usepackage{tocvsec2}

\theoremstyle{plain}
\newtheorem{thm}{Theorem}[section]
\newtheorem{lem}[thm]{Lemma}

\newtheorem{prop}[thm]{Proposition}

\newtheorem{cor}[thm]{Corollary}
\theoremstyle{definition}
\newtheorem{de}[thm]{Definition}

\newtheorem{rem}[thm]{Remark}

\numberwithin{equation}{section}

\def \N {\mathbb N}

\def \Z {\mathbb Z}

\def \O {\mathcal{O}}

\def \F {\mathcal F}
\def \G {\mathcal{G}}
\def \B {\mathcal B}

\def \X {\mathcal{X}}

\def \SS {\mathcal{S}}

\def \PG {{{\bf P}\Gamma}}

\def \a {\alpha }
\def \b {\beta}
\def \ep {\epsilon}

\def \D {\Delta}

\begin{document}

\title{Topological mild mixing of all orders along polynomials}

\author{Yang Cao}

\author{Song Shao}

\address{CAS Wu Wen-Tsun Key Laboratory of Mathematics, and
Department of Mathematics, University of Science and Technology of China, Hefei, Anhui, 230026, P.R. China}

\email{cy412@mail.ustc.edu.cn}

\email{songshao@ustc.edu.cn}



\thanks{This research is supported by NNSF of China (11971455, 11571335).}


\begin{abstract}
A minimal system $(X,T)$ is topologically mildly mixing if all non-empty open subsets $U,V$,
$\{n\in \Z: U\cap T^{-n}V\neq \emptyset\}$ is an IP$^*$-set. In this paper we show that if a minimal system is topologically mildly mixing, then it is mild mixing of all orders along polynomials. That is, suppose that $(X,T)$ is a topologically mildly mixing minimal system, $d\in \N$, $p_1(n),\ldots, p_d(n)$ are integral polynomials with no $p_i$ and no $p_i-p_j$ constant, $1\le i\neq j\le d$, then for all non-empty open subsets $U , V_1, \ldots, V_d $,
$\{n\in \Z: U\cap T^{-p_1(n) }V_1\cap T^{-p_2(n)}V_2\cap \ldots \cap T^{-p_d(n) }V_d \neq \emptyset \}$
is an IP$^*$-set. We also give the theorem for systems under abelian group actions.
\end{abstract}

\maketitle





\section{Introduction}

\subsection{Background}\
\medskip

IP-systems were introduced and studied in topological dynamics by Furstenberg and Weiss in \cite{FW78}, in ergodic theory by Furstenberg in \cite{F,F82}.
Let $S$ be an abelian semigroup and $\N=\{1,2,\ldots\}$ be the set of natural numbers. An {\em IP-set} of $S$ consists of a sequence of (not necessarily distinct) elements $s_i\in S, i\in \N$, together with all products $s_{i_1}s_{i_2}\ldots s_{i_k}$ that can be formed using distinct indices $i_1<i_1<\ldots <i_k$. Let $\F$ denote the family of all non-empty finite subsets of $\N$.
An {\em $\F$-sequence} in $S$ is a sequence $\{s_\a\}_{\a \in \F}$ indexed by the set $\F$. If set $s_\a=s_{i_1}s_{i_2}\cdots s_{i_k}$ for any $\a=\{i_1,i_2,\cdots,i_k\} \in \F$, then any IP-set has form $\{s_\a\}_{\a\in \F}$. Note that if $\a\cap\b=\emptyset $,
then $s_{\a\cup \b}=s_\a s_\b$.
In general $s_\a s_\b$ need not belong to the set and so an IP-set need not be a semigroup. On the other hand every countable semigroup is an IP-set. If each generator $s_i$ of an IP-set occurs infinitely often in the sequence $\{s_i\}_{i=1}^\infty$, then the IP-set generated is a semigroup. Thus IP-sets should be
viewed as generalized semigroups.

For example, let $\{n_i\}_{i=1}^\infty \subseteq \Z$, where $\Z$ is the set of integers. An {\em IP-set} generated by the sequence $\{n_i\}_{i=1}^{\infty}$ is the set
$$FS(\{n_i\}_{i=1}^{\infty}) := \{n_{i_1}+n_{i_2}+\cdots +n_{i_k}:
i_1<i_2<\cdots <i_k \ \text{for some $k\in \N$}\}.$$ If we denote
$n_\a=\sum_{i\in \a}n_i$, then
$FS(\{n_i\}_{i=1}^{\infty})=\{n_\a\}_{\a \in \F}.$
A set is called an {\em IP$^*$-set} if it intersects any IP-set.
An important result about IP-sets is the following Hindman theorem \cite{Hi74}:
For any finite partition of $\N$, one of the cells of the partition contains an IP-set.

\medskip

An IP-set $\{T_\a\}_{\a\in \F}$ of commuting transformations of a space $X$ will be called an {\em IP-system}. The IP-version of van der Waerden's theorem was given by Furstenberg and Weiss in \cite{FW78}, and an ergodic Szemer\'{e}di theorem for IP-systems was built by Furstenberg and Katznelson in \cite{FK85}.
One of mixing properties related to IP-system is mild mixing. The notion of mild mixing was introduced by Furstenberg and Weiss \cite{FW77} (this property was studied independently by Walters \cite{W} from the different viewpoint). Mild mixing is stronger than weak mixing, and weaker than strong mixing \cite{W, FW77}.

\medskip

A measure preserving system $(X,\B,\mu,T)$ is {\em mildly mixing} if for any  $A, B\in  \B$ and any $\ep>0$, $$\{n\in \Z: |\mu(A\cap T^{-n}B)-\mu(A)\mu(B)|<\ep\}$$ is an IP$^*$-set. Furstenebrg showed that mild mixing implies mild mixing of all orders:

\medskip

\noindent {\bf Theorem} {\em (Furstenberg)} \cite{F, F82}
{\em Let $(X,\B,\mu,T)$ be a mildly mixing system and $d\in \N$. Let $A_0,A_1,\ldots, A_d \in \B$ and $a_1,a_2,\ldots, a_d$ be distinct non-zero integers. Then for any $\ep>0$,
$$\{n\in \Z: |\mu(A_0\cap T^{-a_1n}A_1\cap \ldots \cap T^{-a_d n}A_d)-\mu(A_0)\mu(A_1)\ldots \mu(A_d)|<\ep\}$$     is an IP$^*$-set.
}

\medskip


An {\em integral polynomial} is a polynomial taking integer values at the
integers. Bergelson gave a polynomial ergodic theorem for mildly mixing systems.

\medskip

\noindent {\bf Theorem} {\em (Bergelson)}\cite{Bergelson87}
{\em Let $(X,\B,\mu,T)$ be a mildly mixing system and $d\in \N$. Let $p_1(n),\ldots, p_d(n)$ be integral polynomials with no $p_i$ and no $p_i-p_j$ constant, $1\le i\neq j\le d$. Then for all $A_0, A_1, \ldots, A_d\in \B$ and all $\ep>0$,
$$\{n\in \Z: |\mu(A_0\cap T^{-p_1(n)}A_1\cap \ldots \cap T^{-p_d(n)}A_d)-\mu(A_0)\mu(A_1)\ldots \mu(A_d)|<\ep\}$$
is an IP$^*$-set.
}

\medskip


Theorem above was generalized to commuting transformations by Bergelson and McCutcheon \cite{BM00}. The following result is a special case of \cite[Theorem 0.14]{BM00}. Note that in \cite[Theorem 0.14]{BM00} polynomials there are from $\Z^m$ to $\Z$.

\medskip

\noindent {\bf Theorem} {\em (Bergelson-McCutcheon)}\cite{BM00}
{\em Let $(X,\B, \mu, \Gamma)$ be a measure preserving system, where $\Gamma$ is an abelian group such
that for each $T\in \Gamma$, $T\neq e_\Gamma$\footnote{$e_\Gamma$ is the unit of $\Gamma$.}, is mildly mixing. For $d,k\in \N$, let $T_1,\ldots,T_d\in \Gamma$,
and $p_{i,j}, 1\le i\le k, 1\le j\le d$ be integral polynomials
such that the expressions $g_i(n)$
\begin{equation*}
 g_i(n)=T_1^{p_{i,1}(n)}\cdots T_d^{p_{i,d}(n)}, \quad i=1,2,\ldots, k,
\end{equation*}
and the expressions
$$ g_i(n)g_j(n)^{-1}=T_1^{p_{i,1}(n)-p_{j,1}(n)}\cdots T_d^{p_{i,d}(n)-p_{j,d}(n)}, \quad i\neq j\in \{1,2,\ldots, k\},$$
depend nontrivially on $n$\footnote{We say that $g(n)$ depends nontrivially on $n$, if $g(n)$ is a nonconstant mapping of $\Z$ into $\Gamma$.}. Then for all $A_0,A_1,\ldots,A_k\in \B$ and any $\ep>0$,
$$\{n\in \Z: |\mu(A_0\cap g_1(n)^{-1}A_1\cap \ldots \cap g_k(n)^{-1}A_k)-\mu(A_0)\mu(A_1)\ldots \mu(A_k)|<\ep\}$$
is an IP$^*$-set.
}

\medskip


The purpose of this paper is to give the topological version of theorems above.

\subsection{Main results}\
\medskip

The notion of topological mild mixing was introduced by Glasner and Weiss \cite{GW06}, and independently by Huang and Ye \cite{HY04}. A minimal system $(X,T)$ is {\em topologically mildly mixing} if all non-empty open subsets $U,V$,
$$N(U,V)=\{n\in \Z: U\cap T^{-n}V\neq \emptyset\}$$
is an IP$^*$-set.\footnote{This is not the original definition for topological mild mixing in \cite{GW06, HY04}.  There a topological system $(X,T)$ is said to be topologically mildly mixing if the product system with any transitive system is still transitive. But for minimal systems, the definitions are consistent \cite{HY04}.}

First corresponding to Furstenberg's theorem, we have
\begin{thm}\label{Thm-F-top}
Let $(X,T)$ be a topologically mildly mixing minimal system and $d\in \N$. Let $U , V_1, \ldots, V_d $ be non-empty open subsets and $a_1,a_2,\ldots, a_d$ be distinct non-zero integers. Then
$$\{n\in \Z: U\cap T^{-a_1n }V_1\cap T^{-a_2n}V_2\cap \ldots \cap T^{-a_dn }V_d \neq \emptyset \}$$
is an IP$^*$-set.
\end{thm}

As corollary of Theorem \ref{Thm-F-top}, we have:
\begin{cor}\label{gG}
Let $(X,T)$ be a topologically mildly mixing minimal system. Then for any $d\in\N$, any distinct $a_1, \ldots, a_d \in \Z \setminus \{0\}$, and every IP-set $A$, there is a dense $G_\delta$ subset $X_0$ such that for any $x\in X_0$,
\begin{equation*}
  \overline{\{(T^{a_1n}x, \ldots,T^{a_d n}x):n \in A\}}=X^d.
\end{equation*}
\end{cor}

Then we have the following result corresponding to Bergelson's theorem:

\begin{thm}\label{thm-B-top}
Let $(X,T)$ be a topologically mildly mixing minimal system and $d\in \N$. Let $p_1(n),\ldots, p_d(n)$ be integral polynomials with no $p_i$ and no $p_i-p_j$ constant, $1\le i\neq j\le d$. Then for all non-empty open subsets $U , V_1, \ldots, V_d $,
$$\{n\in \Z: U\cap T^{-p_1(n) }V_1\cap T^{-p_2(n)}V_2\cap \ldots \cap T^{-p_d(n) }V_d \neq \emptyset \}$$
is an IP$^*$-set.
\end{thm}

By Theorem \ref{thm-B-top}, we have the following corollaries.

\begin{cor}\label{cor-1.4}
Let $(X,T)$ be a topologically mildly mixing minimal system and $d\in \N$. Let $p_1(n),\ldots, p_d(n)$ be integral polynomials with no $p_i$ and no $p_i-p_j$ constant, $1\le i\neq j\le d$.
Then for every IP-set $A$, there is a dense $G_\delta$ subset $X_0$ such that for any $x\in X_0$,
$$\{(T^{p_1(n)}x, \ldots, T^{p_d(n)}x): n \in A\}$$
is dense in $X^d$.
\end{cor}

\begin{cor}\label{cor-1.5}
Let $(X,T)$ be a topologically mildly mixing minimal system, and let $p(n)$ be a nontrivial integral polynomial. Then for every IP-set $A$, there is a dense $G_\delta$ subset $X_0$ such that for any $x\in X_0$,
$\{T^{p(n)}x: n \in A\}$ is dense in $X$.
\end{cor}

Finally we have a general result for minimal systems under abelian group actions.

\begin{thm}\label{Thm-BM-top}
Let $(X,\Gamma)$ be a minimal system, where $\Gamma$ is an abelian group such
that for each $T\in \Gamma$, $T\neq e_\Gamma$, is topologically mildly mixing minimal. For $d,k\in \N$, let $T_1,\ldots,T_d\in \Gamma$,
and $p_{i,j}, 1\le i\le k, 1\le j\le d$ be integral polynomials
such that the expressions
\begin{equation*}
 g_i(n)=T_1^{p_{i,1}(n)}\cdots T_d^{p_{i,d}(n)}, \quad i=1,2,\ldots, k,
\end{equation*}
and the expressions
$$ g_i(n)g_j(n)^{-1}=T_1^{p_{i,1}(n)-p_{j,1}(n)}\cdots T_d^{p_{i,d}(n)-p_{j,d}(n)}, \quad i\neq j\in \{1,2,\ldots, k\},$$
depend nontrivially on $n$. Then for all non-empty open subsets $U , V_1, \ldots, V_k $,
$$\{n\in \Z: U\cap g_1(n)^{-1} V_1\cap g_2(n)^{-1} V_2\cap \ldots \cap g_{k}^{-1}V_k \neq \emptyset \}$$
is an IP$^*$-set.
\end{thm}

As a corollary, we have:

\begin{cor}\label{cor-1.7}
Under assumptions of Theorem \ref{Thm-BM-top}, for any IP-set $A$ there is a dense $G_\delta$ set $X_0$ of $X$ such that for all $x\in X_0$,
	\begin{equation*}
	\{(g_1(n)x, \ldots , g_k(n)x ): n \in A \}
		\end{equation*}
is dense in $X^k$.
\end{cor}

\medskip

As in \cite{BM00, F82, FK85}, we will deal with mixing for IP-systems, and get results for mild mixing as corollaries. Here we will not repeat the statements of results for mixing IP-systems, and see Section \ref{section-statement} for details.

\subsection{The organization of the paper}\
\medskip

We organize the paper as follows. In Section \ref{section-pre}, we introduce some basic notions and results needed in the paper. In Section \ref{section-statement}, we give the statements of the results of the paper. In Section \ref{section-linear} we give a proof for the linear case to warm up. In Section \ref{section-PET}, we give examples to show how to use PET-induction to prove the results, and in Section \ref{section-proof of Main} we give the complete proof for the main result Theorem \ref{main-thm}.

\section{Preliminaries}\label{section-pre}


\subsection{Topological transformation groups}\
\medskip

A {\em topological dynamical system} (t.d.s. for short) is a triple
$\X=(X, \Gamma, \Pi)$, where $X$ is a compact metric space with metric $\rho$, $\Gamma$ is a
Hausdorff topological group with the unit $e_\Gamma$ and $\Pi: \Gamma\times X\rightarrow X$ is a
continuous map such that $\Pi(e_\Gamma,x)=x$ and
$\Pi(s,\Pi(t,x))=\Pi(st,x)$. We shall fix $\Gamma$ and suppress the
action symbol. Then for any $t\in \Gamma$, $t: X\rightarrow X$ is a homeomorphism, and $e_\Gamma={\rm Id}_X$ is the identity map.

\medskip

Let $(X,\Gamma)$ be a t.d.s. and $x\in X$, then $\O(x,\Gamma)=\{tx: t\in \Gamma\}$ denotes the
{\em orbit} of $x$. A subset
$A\subseteq X$ is called {\em invariant} if $t a\subseteq A$ for all
$a\in A$ and $t\in \Gamma$. When $Y\subseteq X$ is a closed and
$\Gamma$-invariant subset of the system $(X, \Gamma)$, we say that the system
$(Y, \Gamma)$ is a {\em subsystem} of $(X, \Gamma)$. If $(X, \Gamma)$ and $(Y,
\Gamma)$ are two dynamical systems, their {\em product system} is the
system $(X \times Y, \Gamma)$, where $t(x, y) = (tx, ty)$.

\medskip

A system $(X,\Gamma)$ is called {\em minimal} if $X$ contains no proper
non-empty closed invariant subsets. $(X,\Gamma)$ is called {\em transitive} if
for all non-empty open sets $U,V$, there is some $t\in \Gamma$ such that $U\cap t^{-1}V\neq \emptyset$.
It is easy to verify that a system is
minimal iff every orbit is dense. A point $x\in X$ is called a {\em minimal} point if $(\overline{\O(x,\Gamma)}, \Gamma)$ is a minimal subsystem.
A system $(X,\Gamma)$ is {\em weakly mixing} if the product system $(X
\times X,\Gamma)$ is transitive.

When $\Gamma=\Z$, $(X,\Gamma)$ is determined by a homeomorphism $T$, i.e. $T$ is the transformation corresponding to $1$ of $\Z$. In this case, we usually denote $(X,\Z)$ by $(X,T)$.

\subsection{Furstenberg families}\
\medskip

\subsubsection{Furstenberg families} A collection $\G$ of subsets of $\Z$  is  {\em
a (Furstenberg) family} if it is hereditary upward, i.e. $B_1 \subset B_2$ and
$B_1 \in \G$ imply $B_2 \in \G$. A family $\G$ is called {\em
proper} \index{proper family} if it is neither empty nor the entire power set of $\Z$, or,
equivalently if $\Z \in \G$ and $\emptyset \not\in \G$. Any nonempty
collection $\mathcal{A}$ of subsets  of $\Z$ generates a family
$\G(\mathcal{A}) :=\{B \subset \Z: B \supset A$ for some $A \in
\mathcal{A}\}$. \index{$\G(\mathcal{A})$}

For a family $\G$, its {\em dual} \index{dual family} is the family $\G^{\ast}
:=\{B\subset \Z  : B \cap A \neq \emptyset \ \text{for  all} \ A
\in \G \}$. It is not hard to see that $\G^*=\{B \subset\Z:
\Z\setminus B \not \in \G\}$, from which we have that if $\G$ is a
family, then $(\G^*)^*=\G.$

\subsubsection{Filter and Ramsey property}

If a family $\G$ is closed under finite intersections and is proper,
then it is called a {\em filter}. A family $\G$ has the {\em Ramsey property} if $A=A_1\cup A_2\in \G$
implies that $A_1\in \G$ or $A_2\in \G$. It is well known that

\begin{lem}\cite[Lemma 9.4, Lemma 9.5]{F}\label{lem-2.1.}
A proper family has the Ramsey property if and only if its dual $\G^*$ is a filter. Let $\G$ be a family having Ramsey property, then for all $A\in \G$ and $B\in \G^*$, we have $A\cap B\in \G$.
\end{lem}

\subsection{IP systems}\
\medskip

Let $\F$ denote the family of all non-empty finite subsets of $\N$,
i.e. $\a\in \F$ iff $\a=\{i_1,i_2,\cdots, i_k\}\subset \N$, $i_1<i_2<\cdots
<i_k$ for some $k$. For $\a,\b\in \F$, $\a<\b$ (or $\b>\a$) if $\max
\a<\min \b$. A subset of $\F$
$$\F^{(1)}= FU(\{\a_i\}_{i=1}^{\infty}):=\{\bigcup_{i\in
\b}\a_i : \ \b\in \F\}$$ is called {\em IP-ring}, where
$\a_1<\a_2<\cdots<\a_n<\cdots$.

Let $\{n_i\}_{i=1}^\infty \subseteq \Z$. An {\em IP-set} generated by the sequence $\{n_i\}_{i=1}^{\infty}$ is the set
$$FS(\{n_i\}_{i=1}^{\infty}) = \{n_{i_1}+n_{i_2}+\cdots +n_{i_k}:
i_1<i_2<\cdots <i_k \ \text{for some $k\in \N$}\}.$$ If we denote
$n_\a=\sum_{i\in \a}n_i$, then
$$FS(\{n_i\}_{i=1}^{\infty})=\{n_\a\}_{\a \in \F}.$$
Note that we do not require the elements of $\{n_i\}$ are distinct.
A set is called an {\em IP$^*$-set} if it intersects any IP-set. We denote by $\F_{IP}$ and $\F_{IP}^*$ the families generated by all IP-sets and IP$^*$-sets respectively.

\medskip

Here are some equivalent versions of Hindman theorem \cite{Hi74}.

\begin{thm}\cite[Theorem 8.12]{F}
For any finite partition $\F=\bigcup_{i=1}^r C_i $, one of $C_i$
contains an $IP$-ring.
\end{thm}

\begin{thm}\cite[Proposition 9.6]{F}
$\F_{IP}$ has  Ramsey property.
\end{thm}

Let $A=\{n_\a\}_{\a \in \F}$ be an IP-set. A subset $B$ of $A$ is an {\em IP-subset} of $A$ if it is an IP-set itself, i.e. there is some sequence $\{m_i\}_{i=1}^\infty\subset A$ such that $B=\{m_\a\}_{\a\in \F}\subset A$.
It is easy to see that $B$ is an IP-subset of $A$ if and only if there is some IP-ring $\F^{(1)}$ such that
$B=\{n_\a\}_{\a\in \F^{(1)}}$ \cite{F82}.

\begin{de}
Let $A= \{n_\a\}_{\a \in \F}$ be an IP-set. The family generated by all IP-subset of $\{n_\a\}_{\a \in \F}$ is denoted by $\G_A$, that is
\begin{equation*}
\G_A=\{B\subset\Z: B\supset \{n_\a\}_{\a \in \F^{(1)}}, \F^{(1)} \ \text{is an IP-ring}  \}.
\end{equation*}	
\end{de}

\begin{rem}\label{rem-2.5}
\begin{enumerate}
  \item By Hindman theorem, for any IP-set $A= \{n_\a\}_{\a \in \F}$, $\G_A$ has Ramsey property and its dual $\G^*_A$ is a filter. By Lemma \ref{lem-2.1.}, if $A\in \G_A$ and $B\in \G_A^*$, then $A\cap B\in \G_A$.
  \item When $A=\Z$, then $\G_A=\F_{IP}$ and $\G_A^*=\F_{IP}^*$.
\end{enumerate}

\end{rem}

\subsection{IP-mixing and mild mixing}\
\medskip

\begin{de}
Let $(X,\{T_\a\}_{\a\in \F})$ be a IP-system. $(X,\{T_\a\}_{\a\in \F})$  is called {\em mixing} if for all non-empty open subsets $U,V$ of $X$, and any IP-ring $\F^{(1)}$, there is some $\a\in \F^{(1)}$, $$U\cap T^{-1}_\a V\neq \emptyset.$$
\end{de}

\medskip

Let $(X,T)$ be a t.d.s. and let
$\{n_\a\}_{\a\in \F}$ be an IP-set of $\Z$. Let $T_\a=T^{n_{\a} }$
for $\a\in \F$. Then $\{T_\a\}_{\a\in \F}$ is an IP-system.
If $\{T_\a\}_{\a\in \F}$ is mixing as an IP-system, we also say that it is mixing along $\{n_\a\}_{\a\in \F}$. To be precise, we have

\begin{de}
Let $(X,T)$ be a t.d.s. and let $\{n_\a\}_{\a\in \F}$ be an IP-set. $(X,T)$ is said to be {\em mixing along $\{n_\a\}_{\a\in \F}$} if for all non-empty open subsets $U,V$ of $X$, and any IP-ring $\F^{(1)}$, there is some $\a\in \F^{(1)}$, $$U\cap T^{-n_\a} V\neq \emptyset.$$
That is, $$N(U,V)=\{n\in \Z: U\cap T^{-n}V\neq \emptyset\} \in \G_A^*.$$
\end{de}

\begin{rem}
When the IP-set $\{n_\a\}_{\a\in \F}$ is $\Z$, we cover the definition of topological mild mixing. That is,
a minimal system $(X,T)$ is topologically mildly mixing if all non-empty open subsets $U,V$,
$N(U,V)=\{n\in \Z: U\cap T^{-n}V\neq \emptyset\}\in \F_{IP}^*.$
\end{rem}

Since $\G_A^*$ is a filter, we have the following important observation.
\begin{prop}\label{prop2.9}
Let $A=\{n_\a\}_{\a\in \F}$ be an IP-set and $d\in \N$. If $(X,T_1),\ldots, (X,T_d)$ are mixing along $\{n_\a\}_{\a\in \F}$ systems, then
$(X_1\times \ldots\times X_d, T_1\times \ldots \times T_d)$ is mixing along $\{n_\a\}_{\a\in \F}$.

In particular, if $(X,T_1),\ldots, (X,T_d)$ are  topologically mildly mixing minimal systems. Then
$(X_1\times \ldots\times X_d, T_1\times \ldots \times T_d)$ is  topologically mildly mixing.
\end{prop}


For an IP set $\{n_\a\}_{\a\in \F}$, if $(X,T)$ is {mixing along $\{n_\a\}_{\a\in \F}$} and $k\in \Z\setminus\{0\}$, we do not know if $(X,T^k)$ is mixing along $\{n_\a\}_{\a\in \F}$. But it is true for a mildly mixing system.

\begin{prop}\label{prop2.10}
Let $k\in \Z\setminus \{0\}$. If $(X,T)$ is  topologically mildly mixing, then so does $(X,T^k)$.
\end{prop}

\begin{proof}
Let $A=FS(\{n_i\}_{i=1}^{\infty})$ be an IP-set. Then $kA:=FS(\{kn_i\}_{i=1}^{\infty})$ is also an IP-set. Since $(X,T)$ is  topologically mildly mixing, we have $$N(U,V)=\{n\in \Z: U\cap T^{-n}V\neq \emptyset\}\cap kA\neq \emptyset.$$
It follows that
$$\{n\in \Z: U\cap T^{-kn}V\neq \emptyset\}\cap A\neq \emptyset.$$
Thus $(X,T^k)$ is  topologically mildly mixing.
\end{proof}





\subsection{Bergelson-Leibman theorem}\
\medskip

We will need the following Bergelson-Leibman' theorem.

\begin{thm}[Bergelson-Leibman]\cite{Bergelson06}\label{BL}
Let $(X ,\Gamma)$ be a t.d.s. with $\Gamma$ an abelian group, and let $d, k\in \N$. Let $T_1,\ldots ,T_d\in \Gamma$, and let $p_{i,j}$ be integral polynomials with $p_{i,j}(0)=0$, $i=1,2,\ldots,k, j=1,2,\ldots, d$. Then, for any positive $\ep$, there exist $x\in X$ and $n\in \N$ such that
\begin{equation}\label{a11}
 \rho(T_1^{p_{i,1}(n)}T_2^{p_{i,2}(n)}\ldots T_d^{p_{i,d}(n)}x,x)<\ep
\end{equation}
for all $i=1,\ldots ,k$ simultaneously. Moreover, the set
$$\{n\in \Z : \forall \ep >0, \exists x\in X \ \text{ such that} \ \forall i\in \{1,2,\ldots ,k\}, \eqref{a11}\ \text{is satisfied} \}$$
is an IP$^*$-set.
\end{thm}

\subsection{$\{g_1,\ldots, g_k\}$-transitive along $\{n_\a\}_{\a\in \F}$}\
\medskip


Let $({X},\Gamma)$ be a t.d.s., and $A=\{n_\a\}_{\a\in \F}$ be an IP-set.
For $d,k\in \N$, let $T_1,\ldots,T_d\in \Gamma$, and $p_{i,j}(n)$ be integral polynomials, $ 1\le i\le k, 1\le j\le d$ and
\begin{equation*}
  g_i(n)=T_1^{p_{i,1}(n)}\cdots T_d^{p_{i,d}(n)}, \quad i=1,2,\ldots, k.
\end{equation*}

\begin{de}
We call $(X,\Gamma)$ {\em $\{g_1,\ldots, g_k\}$-transitive along $\{n_\a\}_{\a\in \F}$} if for all non-empty open sets $U_1,\ldots, U_d,V_1,\ldots,V_d,\a_0\in \F$,and any IP-ring $\F ^{(1)}$, there is some $\a\in \F ^{(1)}$ such that $\a>\a_0$ and
\begin{equation*}
  U_1\times \ldots \times U_k \cap g_1^{-1}(n_\a)\times \ldots\times g_k^{-1}(n_\a)(V_1\times \ldots \times V_d)\neq \emptyset.
\end{equation*}
\end{de}

The following lemma is a generalization of Lemma 3 of \cite{KO}. For completeness, we include a proof.

\begin{lem}\label{lem-KO} \label{freedom}
Let $(X,\Gamma)$ be a t.d.s. and $T\in\Gamma$. Let $g_1,\ldots, g_k$ be as above.
If $({X}, \Gamma)$ is $\{g_1,\ldots, g_k\}$-transitive along $\{n_\a\}_{\a\in \F}$, then for all non-empty open sets $V_1,\ldots ,V_k$ of $X$, and any IP-ring $\F ^{(1)}$, there is an increasing sequence $\{\a_n\}_{n=0}^\infty\subseteq \F ^{(1)}$ such that $|n_{\a_n}|-n>0$ for all $n$, and for each $i\in \{1,2,\ldots,k\}$, there is a descending sequence $\{V_i^{(n)}\}_{n=0}^\infty$ of open subsets of $V_i$ such that for each $n\ge 0$ one has that
\begin{equation*}
  g_i(n_{\a_j})T^{-j}V_i^{(n)}\subseteq V_i, \quad \text{for all}\quad 0\le j\le n.
\end{equation*}
\end{lem}

\begin{proof}
Let $V_1,\ldots ,V_k$ be non-empty open subsets of $X$ and let $\F^{(1)}$ be an IP-ring. Since $({X},\Gamma)$ is $\{ g_1,\ldots, g_k\}$-transitive along $\{n_\a\}_{\a\in \F}$, there is some $\a_0\in \F ^{(1)}$ such that
\begin{equation*}
  V_1\times \ldots \times V_k \cap g_1^{-1}(n_{\a_0})\times \ldots\times g_k^{-1}(n_{\a_0})(V_1\times \ldots \times V_k)\neq \emptyset.
\end{equation*}
That is, $g_i^{-1}(n_{\a_0})V_i\cap V_i\neq \emptyset$ for all $i=1,\ldots,k$. Put $V_i^{(0)}=g_i^{-1}(n_{\a_0})V_i\cap V_i$ for all $i=1,\ldots,k$ to complete the base step.

Now assume that for $n\ge 1$ we have found a sequence $\a_0<\a_1<\ldots< \a_{n-1}$ of $\F^{(1)}$ and for each $i=1,\ldots,k$, we have non-empty open subsets $V_i\supseteq V_i^{(0)}\supseteq V_i^{(1)}\ldots \supseteq V_i^{(n-1)}$ such that for each $m=0,1,\ldots, n-1$
one has that $|n_{\a_m}|-m>0$ and
\begin{equation}\label{a2}
  g_i(n_{\a_j})T^{-j}V_i^{(m)}\subseteq V_i, \quad \text{for all}\quad 0\le j\le m.
\end{equation}

For $i=1,\ldots, k$, let $U_i=T^{-n}(V_i^{(n-1)})$. Since $({X},\Gamma)$ is $\{ g_1,\ldots, g_k\}$-transitive along $\{n_\a\}_{\a\in \F}$, there is some $\a_n\in \F^{(1)}$such that $\a_n>\a_{n-1}$, $|n_{\a_n}|-n>0$, and
\begin{equation*}
  U_1\times \ldots \times U_k \cap g_1^{-1}(n_{\a_n})\times \ldots\times g_k^{-1}(n_{\a_n})(V_1\times \ldots \times V_k)\neq \emptyset.
\end{equation*}
That is, $g_i^{-1}(n_{\a_n})V_i\cap U_i\neq \emptyset$ for all $i=1,\ldots,k$.

Then for $i=1,\ldots, k$,
\begin{equation*}
  g_i(n_{\a_n})U_i\cap V_i=g_i(n_{\a_n})T^{-n}V_i^{(n-1)}\cap V_i\not=\emptyset.
\end{equation*}
Let
$$V_i^{(n)}=V_i^{(n-1)}\cap \big(g_i(n_{\a_n})T^{-n}\big)^{-1}V_i.$$
Then $V_i^{(n)}\subseteq V_i^{(n-1)}$ is non-empty open set and clearly
$$g_i(n_{\a_n})T^{-n} V_i^{(n)}\subseteq V_i.$$
Since $V_i^{(n)}\subseteq V_i^{(n-1)}$, (\ref{a2}) still holds for $V_i^{(n)}$. Hence we finish our induction. The proof of the lemma is complete.
\end{proof}

\section{Statements of Main results}\label{section-statement}

\subsection{Main Results}\
\medskip

Let $({X},\Gamma)$ be a t.d.s., and $A=\{n_\a\}_{\a\in \F}$ be an IP-set.
For $d,k\in \N$, let $T_1,\ldots,T_d\in \Gamma$, and $p_{i,j}(n)$ be integral polynomials, $ 1\le i\le k, 1\le j\le d$ and
\begin{equation}
  g_i(n)=T_1^{p_{i,1}(n)}\cdots T_d^{p_{i,d}(n)}, \quad i=1,2,\ldots, k.
\end{equation}

\begin{de}
Let $A=\{n_\a\}_{\a\in\F }$ be an IP-set, and let $(X,\Gamma)$ and $g_1,\ldots, g_k$ be defined as above. We say that $(X,\Gamma)$ is {\em $\{g_1,\ldots, g_k\}_\D$-$\G_A^*$-transitive} if for all given non-empty open sets $U, V_1,\ldots, V_k$, $\a_0\in \F$ and IP-ring $\F^{(1)}$, there is $\a\in\F^{(1)}$ such that $\a>\a_0$ and
$$U\cap (g_1(n_\a)^{-1}V_1\cap  \ldots \cap g_k(n_\a)^{-1}V_k)\not=\emptyset,$$
equivalently,
$$\{n\in \Z: U\cap (g_1(n)^{-1}V_1\cap  \ldots \cap g_k(n)^{-1}V_k)\not=\emptyset\}\in \G_A^*.$$
\end{de}	

The main result of this paper is the following:

\begin{thm}\label{main-thm}
Let $A= \{n_\a\}_{\a\in \F}$ be an IP-set and let $(X,\Gamma)$ be a topological system, where $\Gamma$ is an abelian group such that for each $T\in \Gamma$, $T\neq e_\Gamma$, is mixing along $A=\{n_\a\}_{\a\in \F}$ and minimal. For $d,k\in \N$, let $T_1,\ldots,T_d\in \Gamma$,
$\{p_{i,j}(n)\}_{1\le i\le k, 1\le j\le d}$ be integral polynomials
such that the expressions
\begin{equation*}
  g_i(n)=T_1^{p_{i,1}(n)}\cdots T_d^{p_{i,d}(n)}, \quad i=1,2,\ldots, k,
\end{equation*}
depend nontrivially on $n$ for $i=1,2,\dots,k$, and for all $i\neq j\in \{1,2,\ldots,k\}$ the expressions $g_i(n)g_j(n)^{-1}$
depend nontrivially on $n$.
Then $(X,\Gamma)$ is $\{g_1,\ldots, g_k\}_\D$-$\G_A^*$-transitive.
\end{thm}

An immediate consequence of Theorem \ref{main-thm} is as follow.

\begin{thm}\label{main-1}
Let $A=\{n_\a\}_{\a\in \F}$ be an IP-set, $d\in \N$, and let $p_1(n),\ldots, p_d(n)$ be integral polynomials with no $p_i$ and no $p_i-p_j$ constant, $1\le i\neq j\le d$.\footnote{In \cite{Bergelson87}, the polynomials $p(n),q(n)$ are essentially distinct if $p(n)-q(n)\not\equiv {\rm const.}$ If we use this terminology, the condition in the theorem becomes that $p_1,\ldots,p_d$ are pairwise essentially distinct integral polynomials.} Suppose that $(X,T)$ is a minimal system such that  $(X,T^k)$ are mixing along $\{n_\a\}_{\a \in \F}$ for all $k\in \Z\setminus\{0\}$. Then $(X, T)$ is $\{T^{p_1(n)},\ldots, T^{p_d(n)}\}_\D$-$\G_A^*$-transitive.
\end{thm}

When $\{n_\a\}_{\a\in \F}$ is taken as $\Z$, we have Theorem \ref{Thm-BM-top} and Theorem \ref{thm-B-top} respectively.

\subsection{A Lemma}\
\medskip

Before going on, we need the following easy observation.

\begin{lem}\label{diagonal1}
Let $(X,\Gamma)$ be a t.d.s., $\Gamma$ be a group, and $d,k\in \N$. Let $T_1,\ldots,T_d\in \Gamma$,
$\{p_{i,j}(n)\}_{1\le i\le k, 1\le j\le d}$ be integral polynomials, and let
\begin{equation*}
  g_i(n)=T_1^{p_{i,1}(n)}\cdots T_d^{p_{i,d}(n)}, \quad i=1,2,\ldots, k.
\end{equation*}
Let $A\subseteq \Z$ be a sequence.
Then there is a dense $G_\delta$ set $X_0$ of $X$ such that for every $x\in X_0$
\begin{equation*}
  \{(g_1(n)x, g_2(n)x,\ldots,g_k(n)x)\in X^k:n\in A\}
\end{equation*}
is dense in $X^k$ if and only if for all given non-empty open sets $U, V_1,\ldots, V_k$
there is $n\in A$ such that
$$U\cap (g_1(n)^{-1}V_1\cap \ldots \cap g_k(n)^{-1}V_k)\not=\emptyset.$$
\end{lem}

\begin{proof}
One direction is obvious. Now assume that for all given non-empty open sets $U, V_1,\ldots, V_k$ of $X$, there is $n\in A$ such that
$U\cap (g_1(n)^{-1}V_1\cap \ldots \cap g_k(n)^{-1}V_k)\not=\emptyset.$

Let $\mathcal{U}$ be a countable base of $X$, and let
\begin{equation*}
  X_0=\bigcap_{V_1,\ldots,V_k\in \mathcal U}\bigcup_{n\in A} g_1(n)^{-1}V_1\cap \ldots \cap g_k(n)^{-1}V_k.
\end{equation*}
Then it is easy to see that the dense $G_\delta$ subset $X_0$ is what we need.
\end{proof}

\subsection{Corollaries}\
\medskip

By Lemma \ref{diagonal1}, we have following corollaries.

\begin{cor}\label{cor-3.5}
Let $A=\{n_\a\}_{\a\in \F}$ be an IP-set and let $(X,\Gamma)$, $g_1, \ldots, g_k$ be as in Theorem \ref{main-thm}.
Then for any IP-ring $\F^{(1)}$, there is a dense $G_\delta$ set $X_0$ of $X$ such that for any $x\in X_0$
\begin{equation*}
\{(g_1(n_\a)x, \ldots , g_k(n_\a)x ): \a \in \F^{(1)} \}
\end{equation*}
is dense in $X^k$.
\end{cor}

\begin{cor}\label{cor-3.6}
Let $A=\{n_\a\}_{\a\in \F}$ be an IP-set, $d\in \N$, and let $p_1(n),\ldots, p_d(n)$ be integral polynomials with no $p_i$ and no $p_i-p_j$ constant, $1\le i\neq j\le d$. Suppose that $(X,T)$ is a minimal system such that  $(X,T^k)$ are mixing along $\{n_\a\}_{\a \in \F}$ for all $k\in \Z\setminus\{0\}$.
Then for any IP-ring $\mathcal{F}^{(1)}$, there is a dense $G_\delta$-set $X_0$ of $X$ such that for any $x\in X_0$
$$\{(T^{p_1(n_\a)}x, \ldots, T^{p_d(n_\a)}x):\a\in\F^{(1)}\}$$
is dense in $X^d$.
\end{cor}

Thus Corollary \ref{cor-1.4} and Corollary \ref{cor-1.7} follow from Corollary \ref{cor-3.5}, Corollary \ref{cor-3.6} and Proposition \ref{prop2.10}.


\section{Linear case for the commutative actions}\label{section-linear}

In this section we prove the linear case of Theorem \ref{main-thm}. There are two reasons for doing this. One is to show some basic ideas of the proof of Thereom \ref{main-thm} by this special case, and another is that the conditions in the result of this section is a little weaker than that of Theorem \ref{main-thm}.






\medskip

Now we have the main result of this section.

\begin{thm}\label{thm4.1}
Let $(X,\Gamma)$ be a minimal t.d.s with $\Gamma$ abelian, and $d\in \N$. Let $\{T_\a^{(1)}\}_{\a\in \F}$, $\{T_\a^{(2)}\}_{\a\in \F}$, $\ldots$ , $\{T_\a^{(d)}\}_{\a\in \F}$ be mixing IP-systems in $\Gamma$ such that $\{T^{(i)}_\a(T^{(j)}_\a)^{-1}\}_{\a\in \F}$ is mixing, $i\neq j\in \{1,2,\ldots, d\}$ . Then
for all non-empty open sets $U, V_1,\ldots, V_d$ of $X$, and any IP-ring $\F^{(1)}$, there is some $\a\in \F^{(1)}$
$$U\cap (T_{\a}^{(1)})^{-1}V_1\cap \ldots \cap (T_{\a}^{(d)})^{-1}V_d \neq \emptyset.$$
\end{thm}

The proof of Theorem \ref{thm4.1} is the same to the following special case.

\begin{thm}\label{thm-linera-commu}
Let $X$ be a compact metric space and $d\in \N$. Let $T_1, T_2,\ldots, T_d: X\rightarrow X$ be homeomorphisms with $T_iT_j=T_jT_i$ for all $i,j\in \{1,\ldots,d\}$.
Let $A=\{n_\a\}_{\a\in \F}$ be an IP-set. If for all $1\le i\le d$ and $j\neq k\in \{1,\ldots,d\}$, $(X,T_i)$ and $(X, T_jT^{-1}_k)$ are mixing along $\{n_\a\}_{\a\in \F}$ minimal systems, then for all non-empty open sets $U, V_1,\ldots, V_d$ of $X$
\begin{equation*}
  \{n\in \Z: U\cap T_1^{-n}V_1 \cap\ldots\cap T_d^{-n}V_d\neq \emptyset\}\in \G_A^*.
\end{equation*}
\end{thm}

\begin{proof}
We will prove by induction on $d$ that for all non-empty open sets $U, V_1,\ldots, V_d$ of $X$, any IP-ring $\F^{(1)}$, and $\gamma\in \F$, there exists $\a\in \F ^{(1)}$, $\a>\gamma$, such that
\begin{equation*}
  U\cap T_1^{-n_\a}V_1 \cap\ldots\cap T_d^{-n_\a}V_d\neq \emptyset.
\end{equation*}
$d=1$ is trivial. Now we assume that the result holds for $d\ge 1$. Let $U, V_1,\ldots, V_d,V_{d+1}$ be non-empty open subsets of $X$. We will show that for any $\gamma\in \F$ there is some $\a\in \F ^{(1)}$ such that $\a>\gamma$ and
\begin{equation*}
  U\cap T_1^{-n_\a}V_1 \cap\ldots\cap T_d^{-n_\a}V_d\cap T_{d+1}^{-n_\a}V_{d+1}\neq \emptyset.
\end{equation*}
Since $(X,T_1)$ is minimal, there is some $N\in \N$ such that $X=\bigcup_{j=0}^N T_1^{-j}U$.
By Proposition \ref{prop2.9} $(X^{d+1}, T_1\times \ldots \times T_{d+1})$ is mixing along $\{n_\a\}_{\a \in \F}$, and hence it is $\{T_1,\ldots,T_{d+1}\}$-transitive along $\{n_\a\}_{\a\in \F}$.
By Lemma \ref{lem-KO}, there are non-empty subsets $V_1^{(N)},\ldots,V_{d+1}^{(N)}$
and  $\gamma< \a_0< \a_1<\ldots<\a_N$ with $\a_i \in \F ^{(1)}$ for $0\le i\le N$ such that for each $i=1,2,\ldots,d+1$, one has that $ |n_{\a_j}|>j $ and
\begin{equation*}
  T_i^{n_{\a_j}}T_1^{-j}V_i^{(N)}\subseteq V_i, \quad\text{for all}\quad 0\le j\le N.
\end{equation*}

Now applying the induction hypothesis to the system $(X,\langle T_2T_1^{-1},\ldots, T_{d+1}T_1^{-1} \rangle)$ and non-empty subsets $V_1^{(N)},\ldots,V_{d+1}^{(N)}$, there is some $\b\in \F ^{(1)}$ such that $\b>\a_N$ and
\begin{equation*}
  V_1^{(N)}\cap (T_2T_1^{-1})^{-n_\b}V_2^{(N)} \cap\ldots\cap (T_dT_1^{-1})^{-n_\b}V_{d}^{(N)}\cap (T_{d+1}T_1^{-1})^{-n_\b}V_{d+1}^{(N)}\neq \emptyset.
\end{equation*}
Hence there is some $x\in V_1^{(N)}$ such that $(T_iT_1^{-1})^{n_\b}x\in V_{i}^{(N)}$ for $i=2,\ldots,d+1$. Clearly, there is some $y\in X$ such that $T_1^{n_\b} y=x$. Since $X=\bigcup_{j=0}^NT_1^{-j}U$, there is some $j\in \{0,1,\ldots, N\}$ such that $T_1^{j}z=y$ for some $z\in U$. Thus for each $i=1,2,\ldots, d+1$
\begin{equation*}
  \begin{split}
  T_i^{n_\b +n_{\a_j}}z& =T_i^{n_\b +n_{\a_j}}T_1^{-j}y=T_i^{n_\b +n_{\a_j}}T_1^{-j}T_1^{-n_\b}x\\
  &=T_i^{n_{\a_j}}T_1^{-j}(T_iT_1^{-1})^{n_\b} x\in T_i^{n_{\a_j}}T_1^{-j} V_{i}^{(N)}\subseteq V_i.
  \end{split}
\end{equation*}
That is,
$$z\in U\cap T_1^{-n_\a}V_1 \cap\ldots\cap T_d^{-n_\a}V_d\cap T_{d+1}^{-n_\a}V_{d+1},$$
where $\a=\b\cup \a_j\in \F^{(1)}$ as $\b\cap \a_j=\emptyset$.
The proof is complete.
\end{proof}

As a corollary, one has the following:
\begin{cor}\label{cor-gG}
Let $A=\{n_\a\}_{\a\in \F}$ be an IP-set and let $(X,T)$ be a minimal system. Let $U , V_1, \ldots, V_d $ be non-empty open subsets and $a_1,a_2,\ldots, a_d$ be distinct non-zero integers. If for all $k\in \Z\setminus\{0\}$, $(X,T^k)$ are mixing along $A$, then
$$\{n\in \Z: U\cap T^{-a_1n }V_1\cap T^{-a_2n}V_2\cap \ldots \cap T^{-a_dn }V_d \neq \emptyset \}\in \G_A^*.$$
\end{cor}



\section{The PET-induction and some examples}\label{section-PET}

We will use the PET-induction to prove Theorem \ref{main-thm}. The {\em PET-induction} was introduced by Bergelson in \cite{Bergelson87}, where PET stands for {\em polynomial ergodic theorem} or {\em polynomial exhaustion technique} (see \cite{Bergelson87, BM00}). See also \cite{BL96, BL99} for more on PET, and \cite{Leibman94} for PET of nilpotent group actions.

\subsection{Notations}\
\medskip

In the sequal, let $(X,\Gamma)$ be a topological system, where $\Gamma$ is a abelian group generated by $T_1,\ldots,T_d$, $d\in \N$.

An {\em integral polynomial} is a polynomial taking  integer values at the
integers.
Let the $\mathcal{P}$ be the collection of all integral polynomials, i.e. all polynomials taking integer values on the integers and $\mathcal{P}_0$ be the collection of elements $p$ of $\mathcal{P}$ with $p(0)=0$. 


In rest of the paper, we assume that all the polynomials involved have zero constant term, that is, all polynomials are in $\mathcal{P}_0$. It is clear that this assumption can be made without the loss of generality.

\medskip

We will fix the above notations in the rest of the paper.

\subsection{The $\Gamma$-polynomial group}\
\medskip

A $\Gamma$-polynomial $g$ is the one which is represented in the form
\begin{equation}\label{eq-111}
g(n)=\prod_{j=1}^d T_j^{p_j(n)}=T_1^{p_1(n)}T_2^{p_2(n)}\cdots T_d^{p_d(n)},
\end{equation}
where $p_1,\ldots,p_d$ are integral polynomials. For $\Gamma$-polynomials $g(n)=T_1^{p_1(n)}T_2^{p_2(n)}\cdots T_d^{p_d(n)}$ and $h(n)=T_1^{q_1(n)}T_2^{q_2(n)}\cdots T_d^{q_d(n)}$, we define the product and the inverse by
$$gh(n)=T_1^{p_1(n)+q_1(n)}T_2^{p_2(n)+q_2(n)}\cdots T_d^{p_d(n)+q_d(n)},$$
and
$$g^{-1}(n)=T_1^{-p_1(n)}T_2^{-p_2(n)}\cdots T_d^{-p_d(n)}.$$
Then the set of $\Gamma$-polynomials is a group, and it is denoted by $\PG$.
For $T_1^{a_1}T_2^{a_2}\ldots T_d^{a_d}\in \Gamma, (a_1,a_2,\ldots, a_d)\in \Z^d$, its corresponding $\Gamma$-polynomial is $g(n)=T_1^{a_1n}T_2^{a_2n}\ldots T_d^{a_dn}$. And thus
$\Gamma$ itself is
a subgroup of $\PG$ and $e_\Gamma={\rm Id}_X$.

Let
$$\PG_0 = \{g\in \PG : g(n) = T_1^{p_1(n)}T_2^{p_2(n)}\cdots T_d^{p_d(n)}, p_1,p_2,\ldots, p_d\in \mathcal{P}_0\}.$$
And let $\PG_0^*=\PG_0\setminus \{e_\Gamma\}$


\subsection{The weight of $\Gamma$-polynomials}\
\medskip

Let $\Z_+=\{0,1,2,\ldots\}$.
The {\em weight}, $w(g)$, of a $\Gamma$-polynomial $g(n)=\prod_{j=1}^d T_j^{p_j(n)}$ is the
pair $(l,k)$, $l\in \{0,1,\ldots, d\}$, $k\in \Z_+$ for which $p_j=0$ for any $j > l$ and, if
$l\neq 0$, then $p_l\neq 0$ and ${\rm deg}(p_l) = k$. A weight $(l, k)$ is greater than a weight
$(l',k')$, denoted by $(l,k)>(l',k')$, if $l>l'$ or $l = l'$, $k >k'$.

For example, $T_1^n, T_2^{n}, T_1^{n} T_2^{n^3}$ have weights $(1,1), (2,1), (2,3)$ respectively, and $(2,3)> (2,1)>(1,1)$.

Let us now define an equivalence relation on $\PG$: $g(n)=\prod_{j=1}^d T_j^{p_j(n)}$ is equivalent to $h(n)=\prod_{j=1}^d T_j^{q_j(n)}$, if $w(g)=w(h)$ and, if it is $(l,k)$,
the leading coefficients of the polynomials $p_l$ and $q_l$ coincide; we write then
$g\sim h$.
For example,
\begin{equation*}
T_3^{n^2}\sim T_1^{n^6}T_3^{n^2+3n} \sim T_1^{n}T_2^{n^3+3n}T_3^{n^2+5n}.
\end{equation*}
The {\em weight} of an equivalence class is the weight of any of its elements.

\subsection{System and its weight vector}\
\medskip

A {\em system} $\SS$ is a finite subset of $\PG$. For a system $\SS$, if we write $\SS =\{f_i\}_{i=1}^v$
then we require that $f_i\neq f_j$ for $1\le i\neq j\le v$. For every system $\SS$ we define its {\em weight vector}
$\phi(A)$ as follows. Let $w_1<w_2<\ldots <w_q$ be the set of the distinct weights of all equivalence classes
appeared in $\SS$. For $i=1,2,\ldots, q$, let $\phi(w_i)$ be the number of the equivalence classes of elements of $\SS$ with the weight $w_i$.
Let the weight vector $\phi(\SS)$ be
\begin{equation*}
\phi(\SS)=(\phi(w_1)w_1,\phi(w_2)w_2,\ldots, \phi(w_q)w_q).
\end{equation*}

For example, let $\SS=\{T_1^n$, $T_1^{2n}$, $T_1^{3n}$,
$T_1^{n^2}$, $T_1^{n^2+n}$, $T_2^{3n^2+2n}$, $T_1^{n^6}T_2^{3n^2+n}$, $T_1^{n^4+n^3+n}T_2^{3n^2+2n}$, $T_1^{n}T_3^{n^3}$, $T_2^{n^5}T_3^{2n^3+n^2}$, $T_1^{n^2}T_2^{n}T_3^{3n^3+2n^2}\}$. Then
$\phi(\SS)=\big(3(1,1),1(1,2),1(2,2), 3(3,3)\big)$.

\medskip

Let $\SS, \SS'$ be two systems. We say that $\SS'$ {\it precedes} a system $\SS$ if there exists a weight $w$ such
that $\phi(\SS)(w)>\phi(\SS')(w)$ and $\phi(\SS)(u)=\phi(\SS')(u)$ for all weight $u>w$. We denote it by
$\phi(\SS)\succ\phi(\SS')$ or $\phi(\SS')\prec \phi(\SS)$.

For example, let $w_1<w_2<\ldots <w_q$ be a sequence of weights, then
\begin{equation*}
(a_1w_1, \ldots,a_qw_q)\succ (b_1w_1, \ldots, b_qw_q)
\end{equation*}
if and only if $(a_1,\ldots,a_q)>(b_1,\ldots,b_q)$.

\subsection{The PET-induction}\
\medskip

\subsubsection{}
In order to prove that a result holds for all systems $\SS$, we start with the system whose weight
vector is $\{1(1,1)\}$. That is, $\SS =\{T_1^{a_1n}\}$,
where $a_1\in \Z\setminus \{0\}$. Then let $\SS \subseteq \PG_0$ be a system whose weight vector is
greater than $\{1(1,1)\}$, and assume that for all systems $\SS'$ preceding $\SS$, we have that the
result holds for $\SS'$. Once we show that the result still holds for $\SS$, we complete the whole proof.
This procedure is called the {\em PET-induction}.

\subsubsection{}
For example, we outline how to use the language of PET-induction to formulate the proof of Theorem \ref{main-1}.
Now $\Gamma=\langle T\rangle=\{T^n: n\in \Z\}$, and $\PG=\{T^{p(n)}: p\in \mathcal{P}\}$. For each $T^{p(n)}\in \PG$,
its weight $w(T^{p(n)})=(1,k)$, where $k$ is the degree of $p(n)$. A system $\SS$ has the form of $\{ T^{p_1(n)}, T^{p_2(n)},\ldots, T^{p_d(n)} \}$, where $p_1, \ldots, p_d\in \mathcal{P}$ are distinct polynomials. Its weight vector $\phi(\SS)$ has the form of
\begin{equation*}
\big(a_1(1,1), a_2 (1,2),\ldots, a_k (1,k) \big).
\end{equation*}


Under the order of weight vectors, one has
\begin{equation*}
\begin{split}
& \big(1(1,1)\big)\prec\big(2(1,1)\big)\prec\ldots\prec\big(m(1,1)\big)\prec\ldots\prec\big(1(1,2)\big)\prec\big(1(1,1), 1(1,2)\big)\prec\ldots\prec\\
& \big(m(1,1), 1(1,2)\big)\prec\ldots \prec\big( 2(1,2)\big)\prec \big(1(1,1), 2(1,2)\big)\prec\ldots \prec\big(m(1,1), 2(1,2)\big)\prec\ldots\prec\\
&\big(m(1,1), k(1,2)\big)\prec\ldots \prec \big(1(1,3)\big)\prec\big(1(1,1), 1(1,3)\big)\prec\ldots\prec \big(m(1,1), k(1,2), 1(1,3)\big)\\
&\prec \ldots \prec \big(2(1,3)\big)\prec\ldots\prec \big(a_1(1,1), a_2 (1,2),\ldots, a_k (1,k) \big)\prec\ldots.
\end{split}
\end{equation*}

To prove Theorem \ref{main-1}, we will use induction on the weight vectors.
We start from the systems with the weight vector $(1(1,1))$, i.e. $\SS=\{T^{a_1n}\}, a_1\in \Z\setminus \{0\}$.
After that, we assume that the result holds for all systems whose weight vectors are
$\prec $ $\big( a_1(1,1)$, $a_2 (1,2)$, $\ldots$, $a_k (1,k)$ $ \big)$. Then we show
that the  result also holds for the system with weight vector $\big($ $ a_1(1,1),$ $  a_2 (1,2), \ldots, a_k (1,k)$ $ \big)$, and hence the proof is completed.

To illustrate the basic ideas, we show the result for the system $\{T^{an^2+bn}\}$, and the system $\SS=\{T^{n^2}, T^{2n^2}\}$, whose weight vectors are $\big(1 (1, 2) \big)$  and $\big(2 (1, 2) \big)$ respectively.







\subsection{Example 1: $(X,T)$ is $\{T^{an^2+bn}\}_\D$-$\G_A^*$-transitive. }\
\medskip

Let $A=\{n_\a\}_{\a \in \F}$ be an IP-set, and let $(X,T^k)$ be a mixing along $\{n_\a\}_{\a \in \F}$ minimal system for every $k\in \Z \setminus \{0\}$. We show that for all non-empty open subsets $U,V$ of $X$, $$N_{p}(U,V)=\{n\in \Z: U\cap T^{-p(n)}V\neq \emptyset \}\in \G_A^*,$$
where $p(n)=an^2+bn$.
It suffices to show that for all non-empty open subsets $U,V$ of $X$, IP-rings $\F^{(1)}$, and all $\a_0\in \F$, there is some $\a\in \F^{(1)}$ such that $\a>\a_0$ and
$$U\cap T^{-an_\a^2-bn_\a}V\neq \emptyset.$$


\begin{proof}
Since $(X,T)$ is minimal, there is some $N\in \N$ such that $X=\bigcup_{i=0}^N T^{-i}U$. Let $p(n)=an^2+bn$, and $$q(n,m)=p(n+m)-p(m)-p(n)=2anm.$$ 

For each $i=0,1,\ldots, N$, by Theorem \ref{BL} there are some $y_i\in T^{-i}V$ and $\a_i\in \F^{(1)}$ such that
$$T^{p(n_{\a_i})}y_i\in T^{-i}V,\ 0\le i\le N.$$
We may assume that $\a_0<\a_1<\ldots<\a_N$.
Let $V_i$ be a neighborhood of $y_i$ such that
$$T^{p(n_{\a_i})}V_i\subseteq T^{-i}V,\ 0\le i\le N.$$
Now apply Corollary \ref{cor-gG} to $q(n_{\a_0},n)=2a n_{\a_0}n, q(n_{\a_1}, n)= 2an_{\a_1}n,\ldots, q(n_{\a_N},n)=2a n_{\a_N}n$, and there are some $x$ and $\b\in \F^{(1)}$ such that $\b>\a_N$ and
$$T^{q(n_{\a_i}, n_\b)}x=T^{2an_{\a_i}n_\b}x\in V_i, \quad \forall i\in \{0,\ldots,N\}.$$
Then we have that
$$T^{p(n_{\a_i}+n_\b)-p(n_{\b})}x=T^{a(n_{\a_i}+n_\b)^2-an_\b^2+bn_{\a_i}}x =T^{an_{\a_i}^2+bn_{\a_i}+2an_{\a_i}n_\b}x\in T^{p(n_{\a_i})}V_i\subseteq T^{-i}V.$$
Hence
$$T^{i-p(n_\b)}x\in T^{-p(n_{\a_i}+n_\b)}V, \quad \forall i\in \{0,1,\ldots,N\} .$$
Since $X=\bigcup_{i=0}^N T^{-i}U$, there is some $i_0\in \{0,1,\ldots, N\}$ such that
$T^{i_0-p(n_\b)}x\in U$, and thus
$$U\cap T^{-p(n_\a)}V\neq \emptyset,$$
where $\a=\a_{i_0}\cup\b>\a_0$. That is,
$$N_{p}(U,V)\cap \{n_\a\}_{\a\in \F^{(1)}}\neq \emptyset.$$
Since $\{n_\a\}_{\a\in \F^{(1)}}$ is an arbitrary IP-subset of $\{n_\a\}_{\a \in \F}$, it follows that $N_p(U,V)$ is an $\G_A^*$-set.
\end{proof}


\subsection{Example 2: $(X,T)$ is $\{T^{n^2}, T^{2n^2}\}_\D$-$\G_A^*$-transitive.}\
\medskip

Let $A=\{n_\a\}_{\a \in \F}$ be an IP-set, and let $(X,T^k)$ be a mixing along $\{n_\a\}_{\a \in \F}$ minimal system for every $k\in \Z \setminus \{0\}$.
To show this example, we need to verify the following cases one by one:
\medskip

{\it

	\noindent {\bf Case 1:} When the  weight vector is $\big(d(1,1)\big)$:
	$(X,T)$ is $\{T^{a_1n},\ldots,T^{a_dn}\}_\D$-$\G_A^*$-transitive, where $a_1,\ldots, a_d\in \Z\setminus\{0\}$ are distinct integers.
	
	\medskip
	
	\noindent {\bf Case 2:} When the weight vector is $\big(1(1,2)\big)$:
	
	\begin{enumerate}
		\item $(X,T)$ is $\{T^{an^2+b_1n},\ldots,T^{an^2+b_dn}\}$-transitive along $A$,
		\item $(X,T)$ is $\{T^{an^2+b_1n},\ldots,T^{an^2+b_dn}\}_\D$-$\G_A^*$-transitive, where $b_1,\ldots,b_d$ are distinct integers and $a\in \Z\setminus\{0\}$.
	\end{enumerate}
	
	\medskip
	
	\noindent {\bf Case 3:} When the weight vector is $\big(r(1,1), 1(1,2)\big)$:
	
	\begin{enumerate}
		\item $(X,T)$ is $\{T^{c_1n},\ldots,T^{c_rn},T^{an^2+b_1n},\ldots,T^{an^2+b_dn}\}$-transitive along $A$,
		
		\item  $(X,T)$ is $\{T^{c_1n},\ldots,T^{c_rn},T^{an^2+b_1n},\ldots,T^{an^2+b_dn}\}_\D$-$\G_A^*$-transitive, where $a\in \Z\setminus\{0\}$, $b_1,\ldots,b_d$ are distinct integers and $c_1,\ldots,c_r$ are distinct non-zero integers.
	\end{enumerate}
	
	\noindent {\bf Case 4:} When the weight vector is $\big(2(1,2)\big)$:
	\begin{enumerate}
		\item $(X,T)$ is $\{T^{n^2}, T^{2n^2}\}$-transitive along $A$,
		\item $(X,T)$ is $\{T^{n^2}, T^{2n^2}\}_\D$-$\G_A^*$-transitive.
	\end{enumerate}
}

\subsubsection{Case 1: $(X,T)$ is $\{T^{a_1n},\ldots,T^{a_dn}\}_\D$-$\G_A^*$-transitive, where $a_1,\ldots, a_d\in \Z\setminus\{0\}$ are distinct integers.}

\begin{proof}
This is Corollary \ref{cor-gG}.
\end{proof}

\subsubsection{Case 2: $(X,T)$ is $\{T^{an^2+b_1n},\ldots,T^{an^2+b_dn}\}_\D$-$\G_A^*$-transitive, where $b_1,\ldots,b_d$ are distinct integers and $a\in \Z\setminus\{0\}$.}\label{case-00}

\begin{proof}
Let $p_i(n)=an^2+b_in$ for $1 \le i \le d$. First we show that $(X,T)$ is $\{T^{p_1},\ldots, T^{p_d}\}$-transitive along $A$, that is,
for all given open non-empty subsets $U_1,\ldots,U_d, V_1, \ldots, V_d$ of $X$, $\a_0\in \F$, and all IP-rings $\F ^{(1)}$, there is $\a\in \F^{(1)}$ such that $\a>\a_0$ and
$$(U_1\times \ldots \times U_d)\cap (T^{-p_1(n_\a)}V_1\times \ldots \times T^{-p_d(n_\a)}V_d)\not=\emptyset.$$
By Example 1, for $i\in \{1,2,\ldots, d\}$,
$N_{p_i}(U_i,V_i)$ is an $\G_A^*$-set. Thus
$$\{n\in \Z: (U_1\times \ldots \times U_d)\cap (T^{-p_1(n)}V_1\times \ldots \times T^{-p_d(n)}V_d)\not=\emptyset\}=\bigcap_{i=1}^d N_{p_i}(U_i,V_i)$$
is an $\G_A^*$-set as $\G_A^*$ is a filter. In particular, there is $\a\in \F^{(1)}$ such that $\a>\a_0$ and
$$(U_1\times \ldots \times U_d)\cap (T^{-p_1(n_\a)}V_1\times \ldots \times T^{-p_d(n_\a)}V_d)\not=\emptyset.$$

\medskip
To prove the theorem it remains to show for all given non-empty open sets $U, V_1,\ldots, V_d$
$$N_{\{p_1,\ldots, p_d\}}(U,V_1,\ldots, V_d)=\{n: U\cap (T^{-p_1(n)}V_1\cap  \ldots \cap T^{-p_d(n)}V_d)\not=\emptyset\}$$
is an $\G_A^*$-set. To do this, we show that for any IP-ring $\F^{(1)}$,
there is $\a\in\F^{(1)}$ such that
$$U\cap (T^{-p_1(n_\a)}V_1\cap  \ldots \cap T^{-p_d(n_\a)}V_d)\not=\emptyset.$$

Let $\F^{(1)}$ be an IP-ring. And then $\{n_\a\}_{\a\in \F^{(1)}}$ is an IP-subset of $\{n_\a\}_{\a \in \F}$. Assume that $\bigcup_{i=0}^N T^{-i}U=X$ for some $N\in\N$. By Lemma \ref{freedom}, there are increasing sequence $\{\a_j\}_{j=0}^N \subseteq \F^{(1)}$ and $V_i^N\subset V_i$ such that $|n_{\a_j}|>j$ and
$$T^{p_i(n_{\a_j})-j} V_i^{(N)}\subset V_i,\ \  0\le j\le N, 1\le i\le d.$$

Let $$q_i(m,n)=p_i(m+n)-p_i(m)-p_1(n), \quad 1\le i \le d.$$
Applying Case 1 to $q_i(n_{\a_j},n)$ , i.e. $2n_{\a_1}n,\ldots, 2n_{\a_N}n, \ldots, 2n_{\a_1}n+(b_d-b_1)n,\ldots, 2n_{\a_N}n+(b_d-b_1)n$ (since $\{\a_j\}$ is an increasing sequence, we can
choose a subsequence such that the numbers are distinct), then there is $x\in X$ and $\b\in \F^{(1)}$ such that $\b>\a_N$ and
$$T^{q_1(n_{\a_j},n_\b)}x\in V_1^{(N)}, \ldots, T^{q_d(n_{\a_j}, n_\b)}x\in V_d^{(N)}, 0\le j\le N.$$
Let $y=T^{-p_1(n_\b)}x$. Since $X=\bigcup_{i=0}^N T^{-i} U$, there is $z\in U$, $0\le b\le N$ such that $y=T^bz$. Thus $z=T^{-p_1(n_\b)-b}x$. We have
$$T^{p_i(n_\b+ n_{\a_b})}z=T^{q_i(n_{\a_b},n_\b)+p_i(n_{\a_b})-b}x\in T^{p_i(n_{\a_b})-b}V_i^{(N)}\subset V_i,1\le i\le d.$$
That is $$z\in U\cap (T^{-p_1(n_\a)}V_1\cap  \ldots \cap T^{-p_d(n_\a)}V_d)\not=\emptyset,$$ where $\a=\b\cup
\a_b \in \F^{(1)}$ as $\b\cap \a_b= \emptyset$.
Since $\{n_\a\}_{\a\in \F^{(1)}}$ is an arbitrary IP-subset of $\{n_\a\}_{\a \in \F}$, it follows that $N_{\{p_1,\ldots, p_d\}}(U,V_1,\ldots, V_d)$ is an $\G_A^*$-set.
\end{proof}


\subsubsection{Case 3:  $(X,T)$ is $\{T^{c_1n},\ldots,T^{c_rn},T^{an^2+b_1n},\ldots,T^{an^2+b_dn}\}_\D$-$\G_A^*$-transitive, where $a\in \Z\setminus\{0\}$, $b_1,\ldots,b_d$ are distinct integers and $c_1,\ldots,c_r$ are distinct non-zero integers.}\
\medskip

\begin{proof}
Let $p_1(n)=n^2+b_1n, \ldots, p_d(n)=n^2+b_dn$ and $h_1(n)=c_1n, \ldots, h_r(n)=c_rn$.
First similar to Case 2, we have that  $(X,T)$ is $\{T^{p_1},\ldots,T^{ p_d}, T^{h_1},\ldots, T^{h_r}\}$-transitive along $A$, that is, for all given open non-empty subsets $U_1,\ldots,U_{r+d}$, $W_1, \ldots, W_r$, $V_1, \ldots, V_d$ of $X$, and all IP-rings $\F ^{(1)}$, there is $\a\in \F^{(1)}$ such that
$$(U_1\times \ldots \times U_{r+d})\bigcap (T^{-h_1(n_\a)}W_1\times\ldots \times T^{-h_r(n_\a)}W_r)\times(T^{-p_1(n_\a)}V_1\times \ldots \times T^{-p_d(n_\a)}V_d)\not=\emptyset.$$

\medskip

Next we need to show for all given non-empty open sets
$U, W_1, \ldots, W_r, V_1, \ldots, V_d$ of $X$, and all IP-rings $\F ^{(1)}$, there is $\a\in \F^{(1)}$ such that
\begin{equation*}
U\cap (T^{-h_1(n_\a)}W_1\cap\ldots \cap T^{-h_r(n_\a)}W_r)\cap (T^{-p_1(n_\a)}V_1\cap \ldots \cap T^{-p_d(n_\a)}V_d)\not=\emptyset.
\end{equation*}
That is,
\begin{equation}\label{a6}
\begin{split}
   & N_{\{p_1,\ldots, p_d, h_1,\ldots, h_r\}}(U,W_1,\ldots, W_r, V_1,\ldots, V_d) \\
   = & \{n: U\cap (T^{-h_1(n)}W_1\cap\ldots \cap T^{-h_r(n)}W_r)\cap (T^{-p_1(n)}V_1\cap \ldots \cap T^{-p_d(n)}V_d)\not=\emptyset\}
\end{split}
\end{equation}
is an $\G_A^*$-set.

Let $\F^{(1)}$ be an IP-ring. Then $\{n_\a\}_{\a\in \F^{(1)}}$ is an IP-subset of $\{n_\a\}_{\a \in \F}$. Since $(X,T)$ is minimal, there is some $N\in \N$ such that $\bigcup_{i=0}^N T^{-i}U=X$. By Lemma \ref{freedom} there is an increasing suquence $\{\a_j\}_{j=0}^N\subseteq \F^{(1)}$, $W_i^{(N)}\subset W_i$ and $V_s^{(N)}\subset V_s$ such that $|n_{\a_j}|>j$ and
$$T^{h_i(n_{\a_j})-j}W_i^{(N)}\subset W_i,\  \text{and}\ T^{p_s(n_{\a_j})-j}V_s^{(N)}\subset V_s$$ for $0\le j\le N, 1\le i\le r, 1\le s\le d $.

Let $$\widetilde{h}_i(n)=h_{i+1}(n)-h_1(n), \quad 1\le i\le  r-1,$$ and
$$q_s(m,n)=p_s(n+m)-p_s(m)-h_1(n), \quad 1\le s \le d.$$
We will prove by induction. Assume that we have established \eqref{a6} for $r'\le r-1$.
By Case 2 (for $r=1$) or the inductive assumption (for $r\ge 2$) there are $x\in W_1^{(N)}$ and $\b\in\F^{(1)}$ such that $\b>\a_N$ and
$$T^{\widetilde{h}_1(n_\b)}x\in W_2^{(N)}, \ldots, T^{\widetilde{h}_{r'}(n_\b)}x\in W_{r'+1}^{(N)}$$ (we do not have the above line when $r=1$) and
$$T^{q_s(n_{\a_j}, n_\b)}x =T^{p_s(n_{\a_j}+ n_\b)-p_s(n_{\a_j})-h_1(n_\b)}x\in V_s^{(N)}, \ 0\le j\le N, 1\le s\le d.$$
Let $y=T^{-h_{1}(n_\b)}x$. Then by $X=\bigcup_{i=0}^NT^{-i}U$ there is $z\in U$ and $0\le b\le N$ such that $y=T^b z$.
Then $z=T^{-h_1(n_\b)-b}x$ and we have
$$T^{h_i(n_\b+n_{\a_b})}z=T^{h_i(n_\b+n_{\a_b})}T^{-h_1(n_\b)-b}x =T^{\widetilde{h}_{i-1}(n_\b)+h_i(n_{\a_b})-b}x \in T^{h_i(n_{\a_b})-b} W_i^{(N)}\subset W_i$$
for $1\le i\le r'+1$ and
$$T^{p_s(n_\b+n_{\a_b})}z=T^{q_s(n_{\a_b},n_\b)+ p_s(n_{\a_b})-b}x\in T^{p_s(n_{\a_b})-b}V_s^{(N)}\subset  V_s$$
for $1\le s\le d$.
This implies that $$z\in  U\cap (T^{-h_1(n_\a)}W_1\cap\ldots \cap T^{-h_{r'+1}(n_\a)}W_{r'+1})\cap (T^{-p_1(n_\a)}V_1\cap \ldots \cap T^{-p_d(n_\a)}V_d)$$
with $\a=\b\cup \a_b \in \F^{(1)}$ as $\b\cap \a_b=\emptyset$. Since $\{n_\a\}_{\a\in \F^{(1)}}$ is an arbitrary IP-subset of $A$, it follows that $N_{\{p_1,\ldots, p_d, h_1, \ldots, h_r\}} (U,W_1,\ldots, W_r, V_1, \ldots, V_d)$ is an $\G_A^*$-set.
\end{proof}



\subsubsection{Case 4: $(X,T)$ is $\{T^{n^2}, T^{2n^2}\}_\D$-$\G_A^*$-transitive.}

\begin{proof}
Let $p_1(n)=n^2, p_2(n)=2n^2$.
By the same method in the proof of Case 1, $(X,T)$ is $\{T^{p_1}, T^{p_2}\}$-transitive along $A$, that is,
for all given open non-empty subsets $U_1,U_2, V_1, V_2$ of $X$ and IP-rings $\F ^{(1)}$, there is $\a\in \F^{(1)}$ such that
$$(U_1\times U_2)\cap (T^{-p_1(n_\a)}V_1\times T^{-p_2(n_\a)}V_2)\not=\emptyset.$$

\medskip
To prove the theorem we need to show for all given non-empty open sets $U, V_1,V_2$ and IP-rings $\F ^{(1)}$,
there is $\a\in\F^{(1)}$ with
$$U\cap (T^{-p_1(n_\a)}V_1\cap  T^{-p_2(n_\a)}V_2)\not=\emptyset.$$
That is,
\begin{equation*}
N_{\{p_1, p_2\}}(U, V_1,V_2)
   =  \{n: U\cap (T^{-p_1(n)}V_1\cap  T^{-p_2(n)}V_2)\not=\emptyset\}
\end{equation*}
is an $\G_A^*$-set.

Let $\F^{(1)}$ be an IP-ring. Then $\{n_\a\}_{\a\in \F^{(1)}}$ is an IP-subset of $\{n_\a\}_{\a \in \F}$. Let $X=\bigcup_{i=0}^N T^{-i} U$ for some $N\in\N$. By Lemma \ref{freedom} there is an increasing sequence $\{\a_j\}_{j=0}^N\subseteq \F^{(1)}$, $V_1^{(N)}\subset V_1$ and $V_2^{(N)}\subset V_2$ such that $|n_{\a_j}|>j$ and
$$T^{p_1(n_{\a_j})-j}V_1^{(N)}\subset V_1\ \text{and}\ T^{p_2(n_{\a_j})-j}V_2^{(N)}\subset V_2, \quad 0\le j\le N.$$

Let $q_i(m,n)=p_i(n+m)-p_i(m)-p_1(n)$  for $n,m\in \mathbb{Z}$ and $i=1,2$. Since $\{\a_j\}$ is an increasing sequence of $\F^{(1)}$, we have that
all $$q_i(n_{\a_j},n)=\begin{cases} 2n_{\a_j}n \, &\text{ if }i=1 \\
n^2+4n_{\a_j}n \, &\text{ if } i=2 \end{cases}$$
are distinct non-constant polynomials in $n$ for $0\le j\le N, 1\le i\le 2$.

By Case 3, there are $x\in X$ and $\b\in\F^{(1)}$ such that $\b>\a_N$ and
$$T^{q_1(n_{\a_j},n_\b)}x\in V_1^{(N)}, \quad T^{q_2(n_{\a_j},n_\b)} x\in V_2^{(N)},\quad 0\le j\le N.$$
Let $y=T^{-p_1(n_\b)}x$. Then by  $X=\bigcup_{i=0}^N T^{-i} U$ there is $z\in U$, $0\le b\le N$ such that $y=T^bz$. Thus $z=T^{-p_1(n_\b)-b}x$. We have
$$T^{p_1(n_\b+n_{\a_b})}z=T^{q_1(n_{\a_b},n_\b)+p_1(n_{\a_b})-b}x\in T^{p_1(n_{\a_b})-b}V_1^{(N)}\subset V_1$$ and
$$T^{p_2(n_\b+n_{\a_b})}z=T^{q_2(n_{\a_b},n_\b)+ p_2(n_{\a_b})-b}x\in T^{p_2(n_{\a_b})-b}V_2^{(N)}\subset V_2.$$
That is,
$$z\in U\cap (T^{-p_1(n_\a)}V_1\cap  T^{-p_2(n_\a)}V_2)$$
with $\a=\b\cup \a_b\in \F^{(1)}$ as $\b\cap \a_b=\emptyset$.
Since $\{n_\a\}_{\a\in \F^{(1)}}$ is an arbitrary IP-subset, it follows that $N_{\{p_1,p_1\}} (U,V_1,V_2)$ is an $\G_A^*$-set.
\end{proof}

\section{Proof of Theorem \ref{main-thm}}\label{section-proof of Main}

In this section, we give a proof of Theorem \ref{main-thm}.



Let $A= \{n_\a\}_{\a\in \F}$ be an IP-set and let $(X,\Gamma)$ be a topological system, where $\Gamma$ is an abelian group such that for each $T\in \Gamma$, $T\neq e_\Gamma$, is mixing along $A=\{n_\a\}_{\a\in \F}$ and minimal. For $d,k\in \N$, let $T_1,\ldots,T_d\in \Gamma$,
$\{p_{i,j}(n)\}_{1\le i\le k, 1\le j\le d}$ be integral polynomials
such that the expressions
\begin{equation*}
  g_i(n)=T_1^{p_{i,1}(n)}\cdots T_d^{p_{i,d}(n)}, \quad i=1,2,\ldots, k,
\end{equation*}
depend nontrivially on $n$ for $i=1,2,\dots,k$, and for all $i\neq j\in \{1,2,\ldots,k\}$ the expressions $g_i(n)g_j(n)^{-1}$
depend nontrivially on $n$.


\medskip

Using the language introduced in last section, Theorem \ref{main-thm} is restated as follows:

\medskip

\noindent \textbf{Theorem:} \ {\em For any system $\SS=\{g_1,\ldots,g_k\}\subseteq \PG_0^*$,  $(X,\Gamma)$ is $\{g_1,\ldots,g_k\}_\D$-$\G_A^*$-transitive.}



\begin{proof}
We will prove Theorem using the PET-induction. We will use the notations in Section \ref{section-PET} freely.

We start with the system whose weight vector is $\{d(1,1)\}$ with $d\in \N$. That is, $\SS = \{ T_1^{c_1n}, \ldots, T_1^{c_dn}\}$,
where $c_1,\ldots,c_d\in \Z$. This case is Corollary \ref{cor-gG}.

\medskip

Now let $\SS =\{g_1,\ldots,g_k\}\subseteq \PG_0^*$ be a system whose weight vector is greater than $\{d(1,1)\}$ for all $d\in \N$, and assume that for all systems $\SS'$ preceding $\SS$, we have $X$ is $\SS'_\D$-$\G_A^*$-transitive. Now we show that $X$ is $\SS_\D$-$\G_A^*$-transitive.

\medskip

\noindent {\bf Step 1.} \ {\em $(X,\Gamma)$ is $\SS =\{g_1,\ldots,g_k\}$-transitive along $\{n_\a\}_{\a\in \F}$}\
\medskip

Since $\G_A^*$ is a filter, it is sufficient to show that for all $f\in \SS$, and for all given non-empty open subsets $U,V$ of $X$,
$$N_f(U,V):=\{n\in \mathbb{Z}: U\cap f(n)^{-1}V\neq \emptyset\}$$
is an $\G_A^*$-set.

If $f(m+n)=f(n)f(m)$ for all $n,m\in \Z$, then there is $(a_1,a_2,\ldots, a_d)\in \Z^d\setminus\{(0,0,\ldots,0)\}$ such that
$$f(n)=T^{a_1n}_1T_2^{a_2n}\ldots T_d^{a_dn}.$$
In this case $f\in \Gamma$, and by our assumption, $(X,f)$ is mixing along $A$.

Now we assume that $f(m+n)\not \equiv f(n)f(m), n,m\in \Z$.
Let $T\in \Gamma$ be an element of $\Gamma$ with $T\neq e_\Gamma$.
Since $(X,T)$ is minimal, there is some $N\in \N$ such that $X=\bigcup_{i=0}^N T^{-i}U$. Let $h(m,n)=f(m)^{-1}f(m+n)f(n)^{-1}\in \PG_0^*.$ Let $\{n_\a\}_{\a\in \F^{(1)}}$ be an IP-subset of $\{n_\a\}_{\a\in \F}$.

For each $i=1,2,\ldots, N$, by Theorem \ref{BL} there are some $y_i\in T^{-i}V$ and $\a_i\in \F^{(1)}$ such that
$$f(n_{\a_i})y_i\in T^{-i}V,\ 0\le i\le N.$$
We may assume that $\a_0<\a_1<\ldots<\a_N$.
Let $V_i$ be a neighborhood of $y_i$ such that
$$f(n_{\a_i}) V_i\subseteq T^{-i}V,\ 0\le i\le N.$$
Let
$$\SS'=\{h(n_{\a_i}, n): 0\le i\le N\}.$$
Then $\SS'\subset \PG_0^*$ is a system.

Notice that $w(h(n_{\a_i}, n))<w(f)$ for all $i=0,1,\ldots,N$, then we have $\phi(\SS')\prec \phi(\{f\})$. On the other hand, since $f\in \SS$, then $\phi(\{f\}) \preccurlyeq \phi(\SS)$.
Hence $\SS'$ precedes $\SS$.
By the inductive hypothesis, there are some $x$ and $\b\in \F^{(1)}$ such that $\b>\a_N$ and
$$h(n_{\a_i}, n_\b)x \in V_i, \quad \forall i\in \{0,1,\ldots,N\}.$$
Then we have that
$$f(n_{\a_i}+n_\b)f(n_{\b})^{-1}x=f(n_{\a_i})h(n_{\a_i}, n_\b)x\in f(n_{\a_i}) V_i\subseteq T^{-i}V.$$
Hence
$$T^{i}f(n_\b)^{-1}x\in f(n_{\a_i}+n_\b)^{-1}V, \quad \forall i\in \{0,1,\ldots,N\} .$$
Since $X=\bigcup_{i=0}^N T^{-i}U$, there is some $i_0\in \{0,1,\ldots, N\}$ such that
$T^{i_0}f(n_\b)^{-1}x\in U$, and thus
$$U\cap f(n_\a)^{-1} V\neq \emptyset,$$
where $\a=\a_{i_0}\cup\b>\a_0$ as $\a_{i_0}\cap \b=\emptyset$. That is,
$$N_{f}(U,V)\cap \{n_\a\}_{\a\in \F^{(1)}}\neq \emptyset.$$
Since $\{n_\a\}_{\a\in \F^{(1)}}$ is an arbitrary IP-subset of $\{n_\a\}_{\a \in \F}$, it follows that $N_f(U,V)$ is an $\G_A^*$-set.

\medskip

\noindent {\bf Step 2.} \ {\em $X$ is $\{g_1,\ldots,g_k\}_\D$-$\G_A^*$ transitive}\label{claim2}\
\medskip

Let $\SS=\{g_1,\ldots, g_k\}$.
By Lemma \ref{diagonal1} we need to show for all given non-empty open sets
$U, V_1, \ldots, V_v$ of $X$ and all IP-rings $\F ^{(1)}$, there is $\a\in \F^{(1)}$ such that
\begin{equation}\label{thm-4.2-1}
U\cap {g_1(n_\a)^{-1}}V_1\cap \ldots \cap {g_k(n_\a)^{-1}}V_k\not=\emptyset.
\end{equation}
That is,
$$N_{\SS}(U,V_1,\ldots,V_k)=\{n: U\cap {g_1(n_\a)^{-1}}V_1\cap \ldots \cap {g_k(n_\a)^{-1}}V_k\not=\emptyset\}$$
is an $\G_A^*$-set.

\medskip

Let $T\in \Gamma$ be an element of $\Gamma$ with $T\neq e_\Gamma$. As $(X,T)$ is minimal,
there is some $N\in \N$ such that $X=\bigcup_{i=0}^N T^{-i}U$. By Lemma \ref{freedom}, there are $\{\a_j\}_{j=0}^N \subseteq \F^{(1)}$
with $|n_{\a_j}|>j$ for all $j=0,\ldots,N$, and $V_t^{(N)}\subset V_t$ for $t=1,\ldots, k$ such that
\begin{equation}\label{s3}
   g_t(n_{\a_j}) T^{-j}V_t^{(N)}\subset V_t, \quad \forall \ 0\le j\le N.
\end{equation}

\medskip
Let $f\in \SS$ be a $\Gamma$-polynomial of weight minimal in $\SS$:
$w(f)\le w(g_j)$ for any $j=1,\ldots,k$. Without loss of generality assume that $f=g_1$.
Let $$g_{t,j}(n)=g_t(n_{\a_j})^{-1}g_t(n+n_{\a_j})f(n)^{-1}$$ and
\begin{equation*}
\begin{split}
  \SS''=\{g_{t,j}: 1\le t\le k,\ 0\le j\le N \}\setminus \{e_\Gamma\}.
\end{split}
\end{equation*}
Since $\{\a_j\}$ is an increasing sequence, we can choose them such that all elements in $\SS ''$ are distinct.

If $g_t$ is not equivalent to $f$, the $\Gamma$-polynomials $g_{t,0}, \ldots , g_{t,N} \in \SS''$ have the same weights as $g_t$
itself and their equivalence is preserved, that is, if $g_t$ is equivalent to $g_s$ then $g_{t,j}$ is equivalent to $g_{s,i}$ for every $j,i=0,\ldots ,N$. If $g_t$ is equivalent to $f$, then the weights of these $\Gamma$-polynomials decrease: $w(g_{t,j})<w(g_t)=w(f)$. So, the number of equivalence classes having weights greater than $w(f)$ does not change,whereas the number of equivalence classes of $\Gamma$-polynomials having the minimal weight in $\SS$ decreases by 1 when we pass from $\SS$ to $\SS''$.
Then $\SS''$ precedes $\SS$. Notice that if $g_{t,j}=e_\Gamma$, then $t=1$.
By the inductive hypothesis, $X$ is $\SS''_\D$-$\G_A^*$ transitive, and there are $x\in V_1^{(N)}$ and $\b \in \F^{(1)}$ such that $\b>\a_N$ and
$$g_{t,j}(n_\b)x=g_t(n_{\a_j})^{-1}g_t(n_\b+n_{\a_j})f(n_\b)^{-1} x\in V_t^{(N)}, \ 0\le j\le N, 1\le t\le k.$$

Let $y={f(n_\b)^{-1}}x$. Then by $X=\bigcup_{i=0}^NT^{-i}U$ there is $z\in U$ and $0\le b\le N$ such that $y=T^bz$.
Then $z=T^{-b}f(n_\b)^{-1}x$ and we have for each $1\le t\le k$
\begin{equation*}
\begin{split}
{g_t(n_\b+n_{\a_b})}z&=g_t(n_\b+n_{\a_b})T^{-b}f(n_\b)^{-1}x\\ &  = g_t(n_{\a_b})T^{-b}\big(g_t(n_{\a_b})^{-1}g_t(n_\b+n_{\a_b})f(n_\b)^{-1}\big)x\\
& = g_t(n_{\a_b})T^{-b} g_{t,b}(n_\b) x\\ & \in g_t(n_{\a_b})T ^{-b}V_t^{(N)}\subset  V_t.
\end{split}
\end{equation*}
This implies that $$z\in  U\cap {g_1(n_\a)^{-1}}V_1\cap \ldots \cap {g_k(n_\a)^{-1}}V_k$$ with $\a=\b \cup \a_b\in \F^{(1)}$ as $\b \cap \a_b=\emptyset$.
Since $\{n_\a\}_{\a\in \F^{(1)}}$ is an arbitrary IP-subset of $\{n_\a\}_{\a \in \F}$, it follows that $N_\SS (U,V_1,\ldots,V_k)$ is an $\G_A^*$-set.
Hence the proof of the theorem is complete.
\end{proof}

\section{Further discussions}

First recall that a subset $S$ of $\Z$ is {\em syndetic} if it has a bounded gap,
i.e. there is $N\in \N$ such that $\{i,i+1,\cdots,i+N\} \cap S \neq
\emptyset$ for every $i \in {\Z}$. $S$ is {\em thick} if it
contains arbitrarily long runs of integers.
A subset $S$ of $\Z$ is {\em piecewise syndetic} if it is an
intersection of a syndetic set with a thick set. A set $S$ is called {\em thickly syndetic} or {\em replete} if for every $N\in \N$ the positions where
length $N$ runs begin form a syndetic set. Note that the set of all thickly syndetic sets is a filter.  In \cite[Theorem 4.7.]{HY02} it is shown that
for a minimal and weakly mixing system
$(X,T)$, $$N(U,V)=\{n\in\Z: U\cap T^{-n}V\neq \emptyset \}$$ is thickly syndetic for all
nonempty open subsets $U,V$ of $X$.

Weak mixing of all orders was studied in \cite{G94, HSY-19-1}. One of main result of \cite{HSY-19-1} is as follows. Let $(X,\Gamma)$ be a topological system, where $\Gamma$ is a nilpotent group  such that for each $T\in \Gamma$,
$T\neq e_\Gamma$, is weakly mixing and minimal. For $d,k\in \N$ let $T_1,\ldots,T_d\in \Gamma$,
$\{p_{i,j}(n)\}_{1\le i\le k, 1\le j\le d}$ be integral polynomials
such that the expression
\begin{equation*}
  g_i(n)=T_1^{p_{i,1}(n)}\cdots T_d^{p_{i,d}(n)}
\end{equation*}
depends nontrivially on $n$ for $i=1,2,\ldots, k$, and for all $i\neq j\in \{1,2,\ldots,k\}$ the expressions $g_i(n)g_j(n)^{-1}$
depend nontrivially on $n$. Then for all non-empty open sets $U_1,\ldots, U_k$ and $V_1,\ldots,V_k$ of $X$
\begin{equation*}
  \{n\in \mathbb{Z}: U_1\times \ldots \times U_k \cap g_1(n)^{-1}\times \ldots\times g_k(n)^{-1}(V_1\times \ldots \times V_k)\neq \emptyset\}
\end{equation*}
is a thickly syndetic set, and
\begin{equation*}
\{n\in \mathbb{Z}:  U\cap (g_1(n)^{-1}V_1\cap \ldots \cap g_k(n)^{-1}V_k)\not=\emptyset\}
\end{equation*}
is a syndetic set.
We do not know that if the latter is thickly syndetic?

\medskip

Also a minimal system $(X,T)$ is weak mixing if and only if for all non-empty open sets $U,V$, $N(U,V)=\{n\in \Z: U\cap T^{-n}V\neq \emptyset\}$ is lower Banach density 1 \cite{HY04}.
Recall that for a subset $S$ of $\mathbb{Z}$, the {\em upper Banach
density} and {\em lower Banach density} of $S$ are
$$BD^*(S)=\limsup_{|I|\to \infty}\frac{|S\cap I|}{|I|},\ \text{and}\
BD_*=\liminf_{|I|\to \infty}\frac{|S\cap I|}{|I|},$$ where $I$
ranges over intervals of $\mathbb{Z}$.
We also have the following question: under conditions above, is
\begin{equation*}
\{n\in \mathbb{Z}:  U\cap (g_1(n)^{-1}V_1\cap \ldots \cap g_k(n)^{-1}V_k)\not=\emptyset\}
\end{equation*}
lower Banach density 1?



\begin{thebibliography}{SS}







\bibitem{Bergelson87} V. Bergelson, \textit{Weakly mixing PET}, Ergodic Theory Dynam. Systems, {\bf 7} (1987), 337--349.


\bibitem{BL96} V. Bergelson and A. Leibman, \textit{Polynomial extensions of van der Waerden's and
Szemer\'edi's theorems}, Journal of Amer. Math. Soc., {\bf 9} (1996), 725--753.

\bibitem{BL99} V. Bergelson and A. Leibman, \textit{Set-polynomials and polynomial extension of the Hales--Jewett theorem}, Ann. Math., {\bf 150} (1999), 33--75.


\bibitem{BM00} V. Bergelson and  R. McCutcheon, \textit{An ergodic IP polynomial Szemer\'{e}di theorem}, Mem. Amer. Math. Soc., {\bf 146} (2000) 695.

\bibitem{Bergelson06} V. Bergelson, \textit{Combinatorial and Diophantine applications of ergodic theory}, Appendix A by A. Leibman and Appendix B by Anthony Quas and M\'at\'e Wierdl. Handbook of dynamical systems. Vol. 1B, 745--869, Elsevier B. V., Amsterdam, 2006.








\bibitem{FW77} H. Furstenberg and B. Weiss, \textit{The finite multipliers
	of infinite ergodic transformations}, The structure of attractors in
dynamical systems (Proc. Conf., North Dakota State Univ., Fargo,
N.D., 1977), pp. 127--132, Lecture Notes in Math., 668, Springer,
Berlin, 1978.


\bibitem{FW78} H. Furstenberg and B. Weiss, \textit{Topological dynamics and combinatorial number theory}, J. d'Analyse Math., {\bf 34} (1978), 61-85.


\bibitem{F} H. Furstenberg, \textit{Recurrence in ergodic theory and combinatorial number theory},
M. B. Porter Lectures. Princeton University Press, Princeton, N.J., 1981.





\bibitem{F82} H. Furstenberg, \textit{IP-systems in ergodic theory}, Conference
in modern analysis and probability (New Haven, Conn., 1982),
131--148, Contemp. Math., 26, Amer. Math. Soc., Providence, RI,
1984.

\bibitem{FK85} H. Furstenberg and Y. Katznelson, \textit{ An ergodic Szemer\'{e}di theorem
for IP-systems and combinatorial theory}, J. Analyse Math., {\bf  45} (1985), 117--168.






\bibitem{G94} E. Glasner, \textit{Topological ergodic decompositions and applications to products of powers of a minimal transformation}, J. Anal. Math., {\bf 64} (1994), 241--262.


\bibitem{GW06} E. Glasner and B. Weiss, \textit{On the interplay between measurable and topological dynamics}, Handbook of dynamical systems. Vol. 1B, 597--648, Elsevier B. V., Amsterdam, 2006.





\bibitem{Hi74} N. Hindman, \textit{Finite sums from sequences within
cells of a partition of $\N$}, J. Combinatorial Theory Ser. A, {\bf 17} (1974), 1--11.



\bibitem{HY02} W. Huang and X. Ye, \textit{An explicit scattering, non-weakly mixing example and weak disjointness}, Nonlinearity, {\bf 15}(2002), 1--14.

\bibitem{HY04} W. Huang  and {X. Ye}, \textit{Topological complexity, return times and weak disjointness}, Ergod. Th. and Dynam. Sys., {\bf 24} (2004), 825--846.



\bibitem{HSY-19-1} W. Huang, S. Shao and X. Ye, \textit{Topological correspondence of multiple ergodic averages of nilpotent actions}, J. Anal. Math.,  {\bf 138} (2019), no. 2, 687--715.




\bibitem{KO} D. Kwietniak and P. Oprocha, \textit{ On weak mixing, minimality and weak disjointness of all iterates}, Ergodic Theory Dynam. Systems, {\bf 32} (2012), no. 5, 1661-1672.

\bibitem{Leibman94} A. Leibman, \textit{Multiple recurrence theorem for nilpotent group actions}, Geometic and Functional Analysis, {\bf 4}(1994), 648--659.







\bibitem{W} P. Walters, \textit{Some invariant $\sigma$-algebras for measure preserving transformations}, Trans. Amer. Math. Soc., {\bf 163} (1972), 357--368.







\end{thebibliography}
\end{document}